\documentclass[reqno,a4,11pt]{amsart}  

\usepackage[utf8]{inputenc}
\usepackage{gensymb}
\usepackage{amsmath}
\usepackage{amssymb} 
\usepackage{mathtools}   % loads »amsmath«
\usepackage{graphicx}
\usepackage{graphics}  
\usepackage{color}
\usepackage{verbatim} 
\usepackage[normalem]{ulem} 
\usepackage[mathscr]{eucal}
\usepackage{upgreek}
\usepackage{enumerate} 
\usepackage{cite}
\usepackage{amsmath}
\usepackage{amsfonts}
\usepackage{amssymb}
\usepackage{bbm}
\usepackage{cancel}
\usepackage{graphicx}
\usepackage{multicol}
\usepackage[utf8]{inputenc}
\usepackage{framed}
\usepackage{setspace}
\usepackage{amsthm}
\usepackage{hyperref}
\usepackage{caption}
\usepackage{subcaption}
\usepackage{bbm}
\usepackage{graphicx}

\usepackage{mwe}
\usepackage[titletoc]{appendix}
\numberwithin{equation}{section}  
 
\usepackage[refpage,noprefix]{nomencl}
\usepackage{nomencl} 
\newtheorem{theorem}{Theorem}[section] 
 
\newtheorem{proposition}[theorem] {Proposition} 
 
\newtheorem{remark}[theorem]  {Remark}

\theoremstyle{definition}
\newtheorem{example}[theorem] {Example}

 \usepackage{float}
\restylefloat{figure}

\DeclareMathAlphabet{\mathpzc}{OT1}{pzc}{m}{it}

\newcommand{\abs}[1]{\left| #1 \right|}
\DeclarePairedDelimiter{\norm}{\lVert}{\rVert}

\def\d{\delta} 
\newcommand{\e} {\varepsilon}

\def\l{\lambda}

\newfam\Bbbfam 
\font\tenBbb=msbm10 
\font\sevenBbb=msbm7 
\font\fiveBbb=msbm5 
\textfont\Bbbfam=\tenBbb 
\scriptfont\Bbbfam=\sevenBbb 
\scriptscriptfont\Bbbfam=\fiveBbb

\newcommand{\R}     {\mathbb{R}} 
\newcommand{\Z}     {\mathbb{Z}} 
\newcommand{\N}     {\mathbb{N}} 
\renewcommand{\P}   {\mathbb{P}} 
 
\newcommand{\E}     {\mathbb{E}}

 \newcommand{\floor}[1]{\left\lfloor #1 \right\rfloor}

\def\1{{\mathchoice {1\mskip-4mu\mathrm l}      % Blackboard bold 1 
{1\mskip-4mu\mathrm l} 
{1\mskip-4.5mu\mathrm l} {1\mskip-5mu\mathrm l}}} 
 
\def\comment#1{} 
\newtheoremstyle{thm}{2ex}{2ex}{\itshape\rmfamily}{} 
{\bfseries\rmfamily}{}{1.7ex}{} 
 
\newtheoremstyle{rem}{1.3ex}{1.3ex}{\rmfamily}{} 
{\itshape\rmfamily}{}{1.5ex}{} 
 
\newcommand{\Ocal}   {{\mathcal O }}

 \newcommand{\ex}{{\rm e}} 
 
\renewcommand{\d}{{\rm d}}

\newcommand{\Exp}{\mathscr{E}\kern-0.2mm{\operatorname{xp}}}
\newcommand{\Log}{\mathscr{L}\kern-0.2mm{\operatorname{og}}}

\renewcommand{\emptyset} {\varnothing}

\newcommand\NoBlackBoxes{\global\overfullrule0pt}
\NoBlackBoxes

\newcommand{\gk}[1]{\left\{#1\right\}}
\newcommand{\Lh}{\hat{L}}

\renewcommand{\c}{{\mathfrak{c}}}

\newcommand{\Sh}{\widehat{S}}
\newcommand{\et}{\tilde{\e}}
\newcommand{\ek}[1]{\left[#1\right]}
\newcommand{\at}{\tilde{\alpha}}
\newcommand{\rk}[1]{\left(#1\right)}
\newcommand{\hk}[1]{^{(#1)}}
\begin{document}

\title[\hfill Quantitative bounds for large deviations of heavy tails\hfill]
{Quantitative bounds for large deviations of heavy tailed random variables}

\author[Quirin  Vogel]{Quirin  Vogel}
\address[Quirin  Vogel]{NYU Shanghai, 1555 Century Ave, Pudong, Shanghai, China, 200122\newline
Department of Mathematics, CIT, Technische Universität München, Boltzmannstr. 3, D-85748, Garching bei München, Germany.}
\email{quirinvogel@outlook.com}
%\thanks{}

\thanks{}
  
%    \date is required; it is the date received by the editor.
%\date{September 25, 2012}

\subjclass[2010]{Primary: 60F10; Secondary: 60B10}
%    The 2010 edition of the Mathematics Subject Classification is
%    now available.  If you are citing a classification from the
%    new scheme, use the following input coding instead.
%\subjclass[2010]{Primary }
 
\keywords{}  
\begin{abstract}
The probability that the sum of independent, centered, identically distributed, heavy-tailed random variables achieves a very large value is asymptotically equal to the probability that there exists a single summand equalling that value. We quantify the error in this approximation. We furthermore characterise the law of the individual summands, conditioned on the sum being large.
\end{abstract}
%\dedication{Dedication text (use \\[2pt] for line break if necessary)} 
 
 \maketitle
%\tableofcontents
%\begin{center}
%

\section{Introduction and setting}
Large deviation theory concerns the study of random variables taking values away from their mean. A classic result in large deviation theory is that for $S_n=\sum_{i=1}^nX_i$ the sum of i.i.d., centred, integer-valued random variables $(X_i)_i$ with exponential tails, one has that for $x\in\R$
\begin{equation}
    \P\left(S_n>nx\right)=\ex^{-I(x)n\left(1+o(1)\right)}\quad\text{as}\quad n\to\infty.
\end{equation}
Here, $I(x)$ is the Legendre transform of the logarithmic moment generating function of $X_1$, i.e., in this case
\begin{equation}
    I(x)=\sup_{t\ge x}\left\{t x-\log\E\left[\ex^{tX_1}\right]\right\}\, .
\end{equation}
See \cite{dembo2009large} for more details. A follow-up task is the quantification of error- or higher-order terms. A classic result is given in \cite{bahadur1960deviations}, where it is shown that under certain conditions
\begin{equation}
    \P\left(S_n>nx\right)=\frac{\ex^{-I(x)n}}{\sigma\sqrt{n}}\left(1+\Ocal\left(n^{-1}\right)\right)\, ,
\end{equation}
for some $\sigma>0$. Indeed, one often can even give the stronger estimate
\begin{equation}
    \P\left(S_n=nx\right)=\frac{\ex^{-I(x)n}}{\sigma\sqrt{n}}\left(1+\Ocal\left(n^{-1}\right)\right)\, ,
\end{equation}
if $nx$ is in the support of $S_n$, see \cite{blackwell1959probability}. However, when one considers the case where the moment generating function does not exist, the behavior of $\P\left(S_n>nx\right)$ changes drastically. When the tails of $X_i$ decay polynomially (and sufficiently fast), Tchachkuk and Nagaev in \cite{tchachkuk1977limit,nagaev1982asymptotic} show that
\begin{equation}\label{EqIntro1}
     \P\left(S_n>nx\right)=n\P(X_1>nx)\left(1+o(1)\right)\, .
\end{equation}
Recently, Berger in \cite{berger2019notes} gave the improvement
\begin{equation}
     \abs{\frac{\P\left(S_n=nx\right)}{n\P(X_1=nx)}-1}=o(1)\, ,
\end{equation}
given some (mild) local conditions on the tail. There are similar results, for different distributions and cases, see for example \cite{MR2669760,Berger19Reneawal, MR4255236,berger2023collective}. 

Our first result considers the quantification of the error in \cite{berger2019notes}; we show that
\begin{equation}
    \abs{\frac{\P\left(S_n=nx\right)}{n\P(X_1=nx)}-1}=\Ocal\left(\e_n(x)\right)\, ,
\end{equation}
for some vanishing (in many cases explicit) sequence $\e_n(x)$, which depends on the distribution function of the $X_i$'s and on $x$. This is to our best knowledge the first quantification of such error terms in the heavy-tail regime. The Fuk--Nagaev inequality is a vital tool for our analysis, as in other works in this area (see \cite{nagaev1982asymptotic,berger2019notes} for example).\vspace{2mm}\\
Apart from computing the probability of a large deviation event, gaining insight in \textit{how} this deviation is achieved is an important part of large deviation theory. For random variables with existing moment generating function, this often goes by the name \textit{Gibbs-conditioning principle}, see \cite{dembo2009large}. Roughly speaking, the large exceedance is achieved by tilting the distribution of each $X_i$, so that the unlikely value becomes likely in the tilted distribution. The independence is asymptotically preserved. \\
For random variables with sub-exponential tails, the situation is starkly different: the large exceedance is achieved by one of the $X_i$'s assuming the large value, see Equation \eqref{EqIntro1}.

In \cite{armendariz2011conditional}, it was shown that the \textit{total variation} distance between the conditional distribution
\begin{equation}
    \P\left(\{X_i\}_{i=1}^n\in\,\cdot\,\big|S_n>nx\right)\, ,
\end{equation}
and its ``limiting'' distribution converges to zero. The ``limiting'' distribution is defined as follows: independently sample a random variable $Y$ with distribution $\P(Y\in A)=\P(X_1\in A|X_1>nx)$ and $(n-1)$-copies of $X_i$ (according to the original law). A position $i\in\{1,\ldots,n\}$ is sampled uniformly at random. The ``limiting'' law is given by the law of
\begin{equation}\label{eq:limitinglaw}
    \Big(X_1,\ldots,X_{i-1},Y,X_{i},\ldots,X_{n-1}\Big)\, .
\end{equation}
Our contribution to this question is twofold: not only do we quantify the speed of convergence but we also provide a deeper understanding of the conditional law by altering the law of $Y$. In \cite{armendariz2011conditional} the authors give two proofs of their result, one only working for positive random variables and one for the general case. The reason why their first proof breaks down in the general case is that it does not take into account the fluctuations induced by the $(n-1)$-copies of $X_1$. By modifying the law of $Y$, we get a new proof which works in general and also gives the \text{speed} of convergence. \vspace{2mm}\\
Expanding on our previous results, we can also give the limiting law of 
\begin{equation}
    \P\left(\{X_i\}_{i=1}^n\in\,\cdot\,\big|S_n=nx\right)\, .
\end{equation}
This case is interesting as the large value is no longer independent from the $n-1$-copies of $X_i$.

A word regarding the level of generality in this paper: this paper is a compromise between allowing for generality and keeping the notation easy to read. We chose to restrict ourselves to $\Z$-valued random variables with tails consisting of a power-law and a slowly varying function, as in \cite{berger2019notes}. However, similar to \cite{berger2019notes}, the modifications of the arguments (not the notation) needed to address the continuum case ($\R$-valued) are small.

There is a limit to the precision of our local expansion, related to the CLT scale $(a_n)_n$ of the underlying random variables. We introduce the notation
\begin{equation}
    f_n=\omega(g_n)\qquad\text{ if and only if}\qquad o(f_n)=g_n\, ,
\end{equation}
as $n\to\infty$. We furthermore write $f_n\sim g_n$ whenever $f_n=g_n(1+o(1))$, as $n\to\infty$.
\section{Results}
Let $\{X_i\}_i$ be an i.i.d. sequence of $\Z$-valued random variables such that for $x\in\N$
\begin{align}
%    \begin{split}
        \P(X_1=x)&=p\alpha L(x)x^{-(1+\alpha)}\, ,\label{RightTail}\\
        \P(X_1=-x)&=q\alpha L(x)x^{-(1+\alpha)}\, \label{LeftTail}\, ,
%    \end{split}
\end{align}
for $L$ a slowly varying function, $p,q\ge  0$ with $p+q=1$, $\alpha\in (0,\infty)$. If $p=0$, we interpret $p\alpha L(x)x^{-(1+\alpha)}$ as $o\rk{L(x)x^{-(1+\alpha)}}$ and the same for $q=0$.

Recall that $L$ slowly varying means that $L(\l x)\sim L(x)$ for any $\l>0$, as $ x\to\infty$. One may think of $L(x)$ growing/shrinking slower than any polynomial. Note that the mean of $X_1$ exists for $\alpha>1$ and the variance exists for $\alpha>2$.

Suppose that there are two sequence $(a_n)_n$ and$(b_n)_n$ satisfying the following: for $\mu=\E\ek{X_1}$ and $\sigma^2(x)=\E\ek{(X_1-\mu)^2\1\gk{\abs{X_1-\mu}\le x}}$, assume that $\rk{a_n}_n$ satisfies
\begin{align}\label{Equation725}
    \begin{cases}
        L(a_n)(a_n)^{-\alpha} \sim  n^{-1}&\text{ if }\alpha \in (0,2)\, ,\\
        \sigma^2(a_n)a_n^{-2}\sim  n^{-1}&\text{ if }\alpha \ge 2\, ,
    \end{cases}
\end{align}
and that $\rk{b_n}_n$ is given by
\begin{equation}
    b_n=\begin{cases}
    0&\text{ if }\alpha\in (0,1)\, ,\\
    n\E\left[X_1\1\{\abs{X_1}\le a_n\}\right]&\text{ if }\alpha=1\, ,\\
    n\E[X_1]&\text{ if }\alpha>1\, .
    \end{cases}
\end{equation}
Let $S_n=\sum_{i=1}^nX_i$. Then, $S_n$ satisfies a central limit theorem with scales $(a_n)_n$ and $(b_n)_n$, i.e., one has that $\frac{S_n-b_n}{a_n}$ converges to a stable law, see \cite[IX.8, Eq. (8.14)]{felleb1968introduction}. We study the deviations from this central limit theorem.

Finally, we need to quantify how fast the function $L$ varies: we say that $L$ is slowly varying \textit{with precision} $\mathrm{err}[x,y]$ whenever
\begin{equation}
    L\big(x+y\big)=L(x)\left(1+\mathrm{err}[x,y]\right)\quad\text{as }\abs{x}\to \infty\, .
\end{equation}
for $\abs{y}=o(\abs{x})$ and for $o(1)=\mathrm{err}[x,y]$ some function.

Two examples of slowly varying functions are $L(x)=\log(x)^\beta$ and $L(x)=1+\Ocal\left(x^{-\alpha}\right)$, as $x\to\infty$. In the first case, one has that $\mathrm{err}[x,y]\sim \beta y/x$ and in the latter case one has $\mathrm{err}[x,y]=\Ocal\left(x^{-\alpha}\right)$.
\begin{theorem}\label{Thm1}
Suppose that $L$ is slowly varying with precision $\mathrm{err}[x,y]$. Assume Equation \eqref{RightTail} holds with $p>0$ and $\P(X_1<-x)\le \Ocal(1) L(x)x^{-\tilde{\alpha}}$ holds with some $\tilde{\alpha}\ge \alpha$ (as $x\to \infty$). Set $\Sh_n=S_n-\floor{b_n}$ and $\alpha_1=\frac{\alpha}{\alpha+1}\in (0,1)$. Write
\begin{equation}
    A(x,n)= \abs{\frac{\P(\Sh_n=x)}{n\P\left(X_1=x\right)}-1}\, .
\end{equation}
We then have that for every $\e>0$ small enough
\begin{enumerate}
    \item For $\alpha\in (0,2)$, we have that for all $0<x=\omega(a_n)\to\infty$
\begin{equation}
   A(x,n)=\Ocal\left(\left(\frac{a_n}{x}\right)^{(\alpha_1-\e)}+\mathrm{err}\ek{x,(a_n/x)^{\alpha_1}}\right)\, .
\end{equation}
\item For $\alpha=2$, we get that for all $0<x=\omega(a_n\sqrt{\log(n)})\to\infty$
\begin{equation}
   A(x,n)=\Ocal\left(\left(\frac{a_n\sqrt{\log(n)}}{x}\right)^{(\frac{2}{3}-\e)}+\mathrm{err}\ek{x,(a_n/x)^{\alpha_1}}\right)\, .
\end{equation}
\item For $\alpha> 2$, we assume that $x=\omega(\sqrt{n\log(n)})$ as $n\to\infty$. Set $\beta\ge 0$ such that $n^{-\beta}\rk{{x}/{\sqrt{n\log(n)}}}^{1-\alpha_1}\to\infty$ and $\beta\le \frac{(\alpha-2)(\alpha+1)}{2(2\alpha+1)}$. Then
\begin{equation}
     A(x,n)=\Ocal\left(n^{1-\alpha/2+\beta\alpha_1}\left(\frac{\sqrt{n\log(n)}}{x}\right)^{(\alpha_1-\e)}+\mathrm{err}\ek{x,(a_n/x)^{\alpha_1}}\right)\, .
\end{equation}
\end{enumerate}
See Remark \ref{RemarkGaussianCase} for the slightly stronger assumptions in the cases $\alpha\ge 2$.

Note that by symmetry, given Equation \eqref{LeftTail} the theorem also holds true for the limit $x\to-\infty$, with the respective assumption on the right tail. 
\end{theorem}

\begin{example}\label{Ex1}
    If $X_1$ is symmetric \textit{zeta}($1+\alpha$) distributed, i.e., for $k\in \Z\setminus\{0\}$
    \begin{equation}
        \P\left(X_1=k\right)=\frac{|k|^{-(1+\alpha)}}{2\zeta(1+\alpha)}\, .
    \end{equation}
    We then obtain that for all $\alpha>1$, $c>0$, $\e>0$ and for all $x\ge nc$
    \begin{equation}
        \P\left(\Sh_n=x\right)=n \P\left(X_1=x\right)\left(1+\Ocal\left(n^{-\frac{\alpha-1}{\alpha+1}\1\gk{\alpha\le 2}-\frac{\alpha}{2+2\alpha}\1\gk{\alpha>2}+\e}\right)\right)\, ,
    \end{equation}
    as $a_n=n^{\frac{1}{\alpha}\vee\frac{1}{2}}(1+\log(n)\1\gk{\alpha=2})$ and hence $(a_n/n)^{\alpha_1}$ is equal to $n^{-\frac{\alpha-1}{\alpha+1}}$ in the case $\alpha\le 2$ (ignoring the $\log(n)$ factor for $\alpha=2$) and $n^{-\frac{\alpha}{2+2\alpha}}$ for $\alpha>2$. 
    
    Note that for $\alpha>2$, Theorem gives a better error bound, depending on the value of $\beta$. For $\alpha\in (0,1]$, $a_n\ge n$ and hence one needs to choose larger $x$; we leave the details to the reader.
\end{example}
% \begin{remark}
% Theorem \ref{Thm1} gives good asymptotics for a wide class of bounds on $\mathrm{err}[x,y]$. However, even for $\mathrm{err}[x,y]=0$, we retain some \textit{non-avoidable} errors, caused by the fluctuations of the random variables in the limit (Equation \eqref{EqLargeErrorBound}), see also Theorem \ref{Thm2}.
% \end{remark}
Next, we give a non-local version of Theorem \ref{Thm1}.
\begin{theorem}\label{ThmNonLocal}
Suppose that  $\{X_i\}_i$ is an i.i.d. sequence of $\Z$-valued random variables such that for $x\in\N$ and $\Lh(x)$ a slowly varying function
\begin{equation}\label{EquationNonLocalTail}
        \P(X_1\ge x)=p\alpha \Lh(x)x^{-\alpha}\, ,
\end{equation}
and that $\P(X_1\le -x)=\Ocal\rk{1} \Lh(x)x^{-\at}$ for some $\at\ge \alpha$. We then have that for $x$ satisfying the same conditions as in Theorem \ref{Thm1}
\begin{equation}
    \abs{\frac{\P(\Sh_n\ge x)}{n\P\left(X_1\ge x\right)}-1}=\Ocal\left(A(x,n)\right)\, ,
\end{equation}
where $A(x,n)$ is as before.
\end{theorem}
\begin{remark}
    Note that Theorem \ref{Thm1} cannot be deduced from Theorem \ref{ThmNonLocal} as
    \begin{equation}
        \P(X_1=x)=\P(X_1\ge x+1)-\P(X_1\ge x)=p\alpha x^{-\alpha}\rk{\Lh(x+1)\rk{1+1/x}^{-\alpha}-\Lh(x)}\, ,
    \end{equation}
    and $\abs{\Lh(x+1)-\Lh(x)}$ can be much larger than $\Ocal\rk{x^{-1}}$ (take for example $\Lh(x)=1+(-1)^{x}\abs{x}^{-1/2}$).
\end{remark}
As the largest value in the sequence $(X_1,\ldots,X_n)$ could appear at any spot, we introduce the following shift, which moves it to the last spot: let $T\colon \bigcup_{n\in\N}\R^n\to \bigcup_{n\in\N}\R^n$ with (set here $\max\emptyset=-\infty$)
\begin{equation}
    T(x_1,\ldots,x_n)_k=\begin{cases}
    \max_{1\le i\le n}x_i &\text{ when }k=n\, ,\\
    x_n &\text{ when }x_k>\max_{1\le i<k}x_i\text{ and }x_k=\max_{i\ge k}x_i\, ,\\
    x_k&\text{ otherwise}\, .
    \end{cases}
\end{equation}
Denote the law of $X_1$ by $\mu$. Write $F(x)=\mu\left((-\infty,x]\right)$ for the cumulative distribution function and $G(x)=1-F(x)$. Set $\nu_{x,n}=\P\left(\{X_i\}_{i=1,\ldots,n}\in\,\cdot\,|S_n>x\right)$, the distribution of the summands, conditional on $S_n$ large. Let $\nu_x$ be the\footnote{$\nu_x$ does depend on the choice of $\omega(a_n)$. However, this dependence is asymptotically negligible on most events $A$, as $\omega(a_n)=o(x)$.} distribution of $X_1$ conditional on being large:
\begin{equation}
    \nu_x(A)=\P(X_1\in A|X_1>x-\omega(a_n))\qquad\text{for}\qquad\omega(a_n)=o(x)\, ,
\end{equation}
where we recall that $\omega(a_n)$ is any sequence diverging faster than $(a_n)_n$. For the next theorem assume that $0<\omega(a_n)=o(x)$. We use $\norm{\,\cdot\,}$ to denote the \textit{total variation} norm.
\begin{theorem}\label{Thm2}
Assume that $X_1$ has mean zero (or $b_n=0$) and that $G(x+y)/G(x)=1+\mathrm{err}\hk{1}[x,y]$ as $x\to\infty $ and $y=o(x)$. Furthermore, set $c_{n,x}=\abs{\P(S_n\ge x)-n\P\left(X_1\ge x\right)}$. We then have that for $x=\omega(a_n)$ and $x\to\infty$
\begin{equation}\label{Eq:errorThm2}
    \norm{T[\nu_{x,n}]-\mu^{\otimes(n-1)}\otimes \nu_x}^2=\Ocal\left(\max\{\mathrm{err}\hk{1}[x,\omega(a_n)],c_{x,n},nG(x)\}\right)\, .
\end{equation}
In words, we can sample $\{X_i\}_{i=1,\ldots,n}$ conditioned on $S_n>x$ by
\begin{itemize}
\item sampling independently $\{ \tilde X_i\}_{i=1,\ldots,n-1}$ distributed according to $\mu^{\otimes(n-1)}$,
\item a position $i\in \{1,\ldots,n\}$ uniformly,
\item and $Y$ according to $\nu_x$
\end{itemize}
and have the distribution of $\{X_i\}_{i=1,\ldots,n}$ is approximately equal to $\left(\tilde X_1,\ldots,\tilde X_{i-1},Y,\tilde X_{i},\ldots,\tilde X_{n-1}\right)$, with  the error (in total variation norm) given by Equation \eqref{Eq:errorThm2}.
\end{theorem}
\begin{example}
    In the setting of Example \ref{Ex1} with $\alpha=3/2$, we have that for $x>0$ of order $n$
    \begin{equation}
    \norm{T[\nu_{x,n}]-\mu^{\otimes(n-1)}\otimes \nu_x}^2=\Ocal\left(n^{-1/5+\e}\right)\, .
\end{equation}
This allows us to give statements such as: for $A_1,\ldots,A_n$ measurable subsets of $\R$, we get that
\begin{equation}
    \P\left(T[X_1,\ldots,X_n]\in A_1\times\ldots\times A_n|S_n>x\right)\sim \nu_x(A_n)\prod_{i=1}^{n-1}\mu(A_i)\, , 
\end{equation}
as long as the right-hand side has a probability of $\omega\left(n^{-1/10+\e/2}\right)$. This cannot be concluded from the $o(1)$ bounds in \cite{armendariz2011conditional}.
\end{example}
Denote
\begin{equation}
    \xi_{x,n}=\P\left(\{X_i\}_{i=1,\ldots,n}\in\,\cdot\,|S_n=x\right)\, .
\end{equation}
We also set $\xi_{x,n}^*$ the measure given by
\begin{equation}
    \xi_{x,n}^*=\int \d\mu^{\otimes(n-1)}(y)\delta_{x-\sum_{i=1}^{n-1}y_i}\, .
\end{equation}
In words, $\xi_{x,n}^*$ samples $(y_1,\ldots,y_{n-1})$ i.i.d. according to $\mu$ and then sets the final coordinate as $x-\sum_{i=1}^{n-1}y_i$. 
\begin{theorem}\label{Thm3}
Assume that $X_1$ has mean zero (or $b_n=0$) and that $G(x+y)/G(x)=1+\mathrm{err}\hk{2}(x,y)$ as $x\to\infty $ and $y=o(x)$. Set $c_{n,x}=\abs{\P(S_n=x)-n\P\left(X_1=x\right)}$. We then have that for $x=\omega(a_n)$ in the support of $S_n$
\begin{equation}
    \norm{T[ \xi_{x,n}]-\xi_{x,n}^*}^2=\Ocal\left(\max\left\{\mathrm{err}\hk{2}[\omega(x,a_n)],c_{x,n},nG(x)\right\}\right)\quad\text{as}\quad x\to\infty\, .
\end{equation}
\end{theorem}
\section{An application}
In this section, we show how we can use the results above to gain some new insights. Suppose $\left(N_x\right)_{x\in\Z}$ is a collection of independent Poisson random variables with intensity $\l>0$. Consider the random sum
\begin{equation}
    S_n=\sum_{x=-n}^{n}\sum_{i=1}^{N_x}Y_i\hk{x}\, ,
\end{equation}
where $\{Y_i\hk{x}\}_{x\in\Z,i\in\N}$ is a collection of independent symmetric \textit{zeta}($1+\alpha$) distributed random variables, independent of $\left(N_x\right)_{x\in\Z}$. 
\begin{proposition}
Given $\alpha>1$, for any $c>0$, uniformly in $k\ge cn$
\begin{equation}
    \P(S_n=k)=\P\left(\exists x\in \{-n,\ldots, n\}\text{ and }i\in\{1,\ldots,N_x\}\colon Y_i\hk{x}=k \right)\left(1+\Ocal\left(n^{-\beta}\right)\right)\, ,
\end{equation}
for some $\beta=\beta_\alpha>0$. This can be interpreted as a condensation phenomena, see \cite{grosskinsky2003condensation}. The constant $\beta$ is the same as in Example \ref{Ex1}.
\end{proposition}
In \cite{kluppelberg1997large} the asymptotics of the cumulative distribution function $\P(S_n>k)$ were obtained, however neither the error term was quantified nor the probability density function approximated.
\begin{proof}
The idea is that the parameter $n$ in Theorem \ref{Thm1} is now Poisson distributed with parameter $(2n+1)\l$. However, by standard large deviation estimates for Poisson random variables, one can show that such a Poisson random variable is bounded by $(2n+1)\l\pm n^{1/2+\e}$ for any $\e>0$, outside a set of stretch exponentially small probability. Hence, we can apply Theorem \ref{Thm1}.

Note that for every $\e>0$ there exists a $\delta>0$ such that 
\begin{equation}
    \P\left(\abs{\sum_{x=-n}^{n}N_x-(2n+1)\l}\ge n^{1/2+\e}\right)=\Ocal\left(\ex^{-n^\delta}\right)\, ,
\end{equation}
see \cite[Eq. (2.2.12)]{dembo2009large}.

Conditional on the value of $\sum_{x=-n}^{n}N_x$ and on the event $\abs{\sum_{x=-n}^{n}N_x-(2n+1)\l}\le n^{1/2+\e}$, we can apply Theorem \ref{Thm1} to get
\begin{equation}
    \P\left(\sum_{x=-n}^{n}\sum_{i=1}^{N_x}Y_i\hk{x}=k\Big| \sum_{x=-n}^{n}N_x\right)=\P(Y_1\hk{0}=k)\sum_{x=-n}^{n}N_x\left(1+\Ocal\left(n^{-\beta}\right)\right)\, ,
\end{equation}
for some $\beta>0$. Furthermore, note that on the event $\gk{\abs{\sum_{x=-n}^{n}N_x-(2n+1)\l}\le n^{1/2+\e}}$ and for $M=\sum_{x=-n}^{n}N_x$
\begin{multline}
    \P\left(\exists (x,i)\in \{-n,\ldots, n\}\times\{1,\ldots,N_x\}\colon Y_i\hk{x}=k\Big|M\right)\\
    \sim M\P\left(Y_1\hk{0}=k\right)\rk{1+\Ocal\rk{n\P\left(Y_1\hk{0}=k\right)}}\, ,
\end{multline}
by the fundamental property of Poisson processes.
\end{proof}
\section{Proofs}
\subsection{Technical preliminaries}
Before embarking on the proof, we recall the scales involved in our analysis:
\begin{itemize}
    \item The scale $n$, given.
    \item The scale $a_n$, induced by the CLT scaling. It satisfies $L(a_n)a_n^{-\alpha}\sim n^{-1}$ if $\alpha\in(0,2)$ and $\sigma^2(a_n)a_n^{-2}\sim n^{-1}$ if $\alpha\ge 2$, see Equation \eqref{Equation725}.
    \item The scale of $x$. It only has to obey the constraint that $x=\omega(a_n)$.
    \item The induced scale $\frac{x}{a_n}$. It relates to the best possible error we can achieve.
\end{itemize}
Recall Potter's bound (see \cite[Theorem 1.5.6]{bingham1989regular}) which gives for $L$ slowly varying and any $\delta>0$, that there exists $c_\delta$ such that for $a,b$ sufficiently large
\begin{equation}\label{eqrefpot}
    L(a)/L(b)\le c_\delta \max\{(a/b)^\delta,(b/a)^\delta\}\, .
\end{equation}
% Using the above asymptotical relation for $a_n$, we have that
% \begin{equation}\label{Equation216}
%     nL(x)x^{-(1+\alpha)}\sim \left(\frac{x}{a_n}\right)^{-\alpha}\!\!\!\cdot x^{-1}\frac{L(x)}{L(a_n)}=\frac{1}{x}\Ocal\left(\left(\frac{x}{a_n}\right)^{-\alpha+\delta}\right)\, ,
% \end{equation}
% for any $\delta>0$.
\begin{remark}[The Gaussian domain of attraction]\label{RemarkGaussianCase}
    For $\alpha\ge 2$, the limiting law of $(S_n-b_n)/a_n$ is Gaussian. This changes the big jump phenomenon of $\P\rk{S_n=x}$ in the region where $a_n\le x\le Ca_n\log(n)$, for $C>0$. This was already observed by Nagaev \cite{nagaev1982asymptotic} in the case $\alpha>2$ and $\gk{S_n\ge x}$, see the recent \cite{berger2023collective} for the complete picture. We summarize the points relevant to our case: if $\alpha>2$, the have that 
    \begin{equation}
        \P\rk{S_n-b_n>x}\sim pnL(x)x^{-\alpha}\qquad\text{if}\qquad x>b\rk{n\log(n)}^{1/2}\, ,
    \end{equation}
    where $b>(\alpha-2)^{1/2}$. If $b<(\alpha-2)^{1/2}$, this is no longer true (for the case $b=(\alpha-2)^{1/2}$, see \cite{berger2023collective}). 
    
    If $\alpha=2$, we need to be more careful: set $q(x)=x^2\P(X_1>x)/\sigma^2(x)$. For $\alpha=2$, we have that $\sigma^2(x)$ is slowly varying and grows faster than $L(x)$ (see \cite[Proposition 1.5.9a]{bingham1989regular}) and hence $q(x)=o(1)$, as $x\to\infty$, and slowly varying. Then, using \cite[Equation 2.9]{berger2023collective}, we have the occurrence of the single big jump if
    \begin{equation}
        \liminf_{n\to \infty}\frac{\rk{x/(a_n)}^2}{2\log q(a_n)}>1\, ,
    \end{equation}
    and no big jumps if the limsup is bounded from above by 1. By the Potter bounds for any $\e>0$, $q(x)=\Ocal\rk{x^{-\e}}$ as $x\to\infty$ and hence if $(x/a_n)^2$ grows faster than $\sqrt{\log(n)}$, a big jump will occur.
\end{remark}

We recall the \textit{local Fuk--Nagaev inequality} from \cite[Theorem 5.1]{berger2019notes}.
\begin{theorem}
Fix $\alpha>0$. Set $M_n$ the maximum of the $X_1,\ldots,X_n$. Write $\Sh_n$ for the recentered walk $\Sh_n=S_n-\floor{b_n}$. Write $\sigma_{2,\alpha}(x)=\E\ek{\abs{X_1}^{\alpha}\1\gk{\abs{X_1}>x}}$ if the tails decay with speed $\alpha>2$. Again, $\sigma_{2,\alpha}(x)$ is slowly varying, see \cite[Proposition 1.5.9a]{bingham1989regular}. Under the conditions from Theorem \ref{Thm1}, there exist $\c_1,\c_2,\c_3>0$ such that for every $1\le y\le x$ and every $x$ with $x\ge a_n$, we have
\begin{equation}\label{EqFukNag}
    \P\left(\Sh_n=x,\,M_n\le y\right)\le\frac{\c_3}{a_n}\begin{cases}
    \ex^{-\c_1x^2/n} +\left(\frac{xy^{\alpha-1}}{n\sigma_{2,\alpha}(y)}\right)^{-\c_2x/y}&\textnormal{ if }\alpha> 2\, ,\\
    \ex^{x/y}\rk{1+\frac{xy}{n\sigma_2(y)}}^{-x/y}&\textnormal{ if }\alpha=2\, ,\\
    \left(\c_1\frac{y}{x}nL(y)y^{-\alpha}\right)^{x/2y}&\textnormal{ if }\alpha\in (1,2)\, ,\\
    \ex^{\frac{3x}{y}}\left(1+\frac{\c_1x}{nL(y)}\right)^{-\frac{x}{4y}}+\ex^{-\c_2(x/a_n)^2}&\textnormal{ if }\alpha=1\, ,\\
    \ex^{x/4y}\left(1+\frac{\c_1 x}{ny^{1-\alpha}L(y)}\right)^{-x/y}&\textnormal{ if }\alpha<1\, .
    \end{cases}
\end{equation}
\end{theorem}
\begin{proof}
    The above result is stated in \cite[Theorem 5.1]{berger2019notes} for the case $\alpha\in(0,2)$. For $\alpha> 2$, it follows from \cite[Corollary 1.7]{MR0542129}. For $\alpha=2$, it follows from \cite[Lemma 5.2]{berger2023collective}.
\end{proof}

To ease reading, we write for $a,b\in \R$ and any $f\colon\Z\to\R$
\begin{equation}
    \sum_{k=a}^bf(k):=\sum_{k\in [a,b]\cap\Z}f(k)\, ,\qquad\textnormal{and similarly for}\qquad \sum_{k\ge a}f(k)\quad\text{and}\quad\sum_{k\le a}f(k)\, .
\end{equation}
\subsection{Proof of Theorem \ref{Thm1}}
We now begin with the main proof: without loss of generality, assume that $x>0$. Fix a sequence $\e_n=o(1)$ large enough such that $\e_nx/a_n\to+\infty$ (for $\alpha\in(0,2)$) and $\e_nx/(a_n\log^{1/2}(n))\to+\infty$ for $\alpha\ge 2$. The sequence $\e_n$ allows us to interpolate between the CLT scale $(a_n)_n$ and speed of divergence of $x$. We also fix the sequence $\et_n=\log(x/a_n\sqrt{\log(n)})^{-2}=o(1)$. Note that $\et_n x\sim x/\log^2(x)\to\infty$.

In the first part of the proof, we give general error bounds, valid for (almost) all kind of remainders $\mathrm{err}[x,y]$. In the second part of the proof, we collect all the errors and simplify. This allows to adapt the result easily to all type of error estimates without overloading the notation. We expand
\begin{equation}
    \begin{split}
        \P(\Sh_n=x)=&\P\left(\Sh_n=x, M_n\ge (1-\e_n)x\right)\\
        &+\P\left(\Sh_n=x, M_n\in (\et_nx,[1-\e_n]x)\right)+\P\left(\Sh_n=x, M_n\le \et_n x\right)\\
        =&\mathrm{A}+\mathrm{B}+\mathrm{C}\, .
    \end{split}
\end{equation}
\textbf{1. Estimating $\mathrm{A}$}: we begin by dissecting $\mathrm{A}$
\begin{equation}
  \begin{split}
      A= \P\left(\Sh_n=x,(1-\e_n)x\le M_n\le (1+\e_n)x\right)+\P\left(\Sh_n=x, M_n>(1+\e_n)x\right)=\mathrm{A1+A2}\, .
    %   =&\sum_{y\ge (1-\e_n)x}\P\left(\Sh_n=x, M_n=y\right)\\
    %   &\le \sum_{y\ge (1-\e_n)x}n\P\left(X_1=y\right)\P\left(S_{n-1}-\floor{b_n}=x-y\right)\, ,
  \end{split}  
\end{equation}
The second term is negligible, as we will see later. For the first term, we write
\begin{equation}
   \begin{split}
        \mathrm{A1}=\P\left(\Sh_n=x,(1-\e_n)x\le M_n\le (1+\e_n)x\right)&=\sum_{y= (1-\e_n)x}^{(1+\e_n)x}\P\left(\Sh_n=x, M_n=y\right)\, .
   \end{split}
\end{equation}
We begin with an upper bound
\begin{equation}
    \sum_{y= (1-\e_n)x}^{(1+\e_n)x}\P\left(\Sh_n=x, M_n=y\right)\le \sum_{y= (1-\e_n)x}^{(1+\e_n)x}n\P\left(X_1=y\right)\P\left(S_{n-1}-\floor{b_n}=x-y\right)\, ,
\end{equation}
where we used the independence and a union bound.

Fix $y\in x[1-\e_n,1+\e_n]$ and write $w=y-x$. We then have that
\begin{equation}
    \P\left(X_1=y\right)=p\alpha L(y)y^{-(1+\alpha)}=p\alpha L(x+w)(x+w)^{-(1+\alpha)}\, .
\end{equation}
Using the error bounds we have for $L$ and the binomial series, we get
\begin{equation}
    \P\left(X_1=y\right)=p\alpha L(x)x^{-(1+\alpha)}\left(1+\mathrm{err}[x,\e_nx]+\Ocal(\e_n)\right)\, .
\end{equation}
Therefore,
\begin{equation}
     \P\left(\Sh_n=x,(1-\e_n)x\le M_n\le (1+\e_n)x\right)\le p\alpha n L(x)x^{-(1+\alpha)}\left(1+\mathrm{err}[\e_n ,x]+\Ocal(\e_n)\right)\, .
\end{equation}
On the other hand, we have
\begin{equation}\label{EquationOmitted}
\begin{split}
    \sum_{y= (1-\e_n)x}^{(1+\e_n)x}\P\left(\Sh_n=x, M_n=y\right)\ge &\sum_{y= (1-\e_n)x}^{(1+\e_n)x}n\P\left(X_1=y\right)\P\left(S_{n-1}-\floor{b_n}=x-y\right)\\
    &-\frac{1}{2}\sum_{y= (1-\e_n)x}^{(1+\e_n)x}n(n-1)\P\left(X_1=y\right)^2\P\left(S_{n-2}-\floor{b_n}=x-y\right)\, .
\end{split}
\end{equation}
As above, the first sum is $p\alpha n L(x)x^{-(1+\alpha)}\left(1+\mathrm{err}[\e_n ,x]+\Ocal(\e_n)\right)$. The second sum is bounded by
\begin{equation}
    Cn^2 L^2(x)x^{-2(1+\alpha)}\, ,
\end{equation}
for some $C>0$ and is negligible as we will see later.

For the second term $\mathrm{A2}$, we have
\begin{equation}\label{Equation1today}
    \begin{split}
       \mathrm{A2}=\P\left(\Sh_n=x, M_n>(1+\e_n)x\right)=&\sum_{y\ge (1+\e_n)x}\P\left(\Sh_n=x, M_n=y\right)\\
         &\le \sum_{y\ge (1+\e_n)x}n\P\left(X_1=y\right)\P\left(S_{n-1}-\floor{b_n}=x-y\right)\, ,
    \end{split}
\end{equation}
where we again used the exchangeability of the $X_i$'s and a union bound. We can estimate the first term by its maximum to conclude
\begin{equation}\label{Equation2today}
      \P\left(\Sh_n=x, M_n>(1+\e_n)x\right)\le n\sup_{y\ge (1+\e_n)x}\P\left(X_1=y\right)\P\left(S_{n-1}-\floor{b_n}\le -\e_n x\right)\, .
\end{equation}
Recall the condition on $\e_n x$ stated at the beginning of the proof and that the left tails of $X_1$ decay with speed at least $\Ocal\rk{L(x)x^{-\alpha}}$. We have that for some $C>0$
\begin{equation}\label{Equation711231}
    \P\left(S_{n-1}-\floor{b_n}\le -\e_n x\right)\le C n L(\e_nx)(\e_nx)^{-\alpha}\, ,
\end{equation}
using \cite[Theorem 2.1]{berger2019notes} in the case $\alpha\in(0,2)$, \cite[Theorem 2]{doney2001local} for $\alpha>2$ and \cite[Equation 2.33]{berger2023collective} for the case $\alpha=2$. Hence
\begin{equation}
    \P\left(\Sh_n=x, M_n>(1+\e_n)x\right)=\Ocal\left(x^{-(1+\alpha)}n^2 L(x)L(\e_nx)(\e_nx)^{-\alpha}\right)\, .
\end{equation}
To summarize: we have that
\begin{equation}
    \mathrm{A}= p\alpha n L(x)x^{-(1+\alpha)}\left(1+\mathrm{err}[x,\e_n x]+\Ocal(\e_n)+nL(\e_nx)(\e_nx)^{-\alpha}\right)\, .
\end{equation}
\textbf{2. Estimating $\mathrm{B}$}: to bound the term $\mathrm{B}$, we expand, as in Equations  \eqref{Equation1today} and \eqref{Equation2today}, for some $C_1'>0$ universal 
\begin{equation}
\begin{split}
    \mathrm{B}&=\P\left(\Sh_n=x, M_n\in (\et_nx,[1-\e_n]x)\right)=\sum_{y={\et_nx}+1}^{{(1-\e_n)x}-1}\P\left(\Sh_n=x, M_n=y\right)\\
    &\le \sum_{y={\et_nx}+1}^{{(1-\e_n)x}-1}\P\left(X_1=y\right)\P\left(S_{n-1}-\floor{b_n}=x-y\right)\\
    &\le \rk{\sup_{y\in \ek{{\et_nx}+1,{(1-\e_n)x}-1}}\P\left(X_1=y\right)}\P\left(S_{n-1}-\floor{b_n}\ge \e_n x\right)\\
    &\le C_1' nL(x)x^{-(1+\alpha)}\et_n^{-(1+\alpha)}\frac{L(\et_n x)}{L(x)}\P\left(S_{n-1}-\floor{b_n}\ge \e_n x\right)\, .
    \end{split}
\end{equation}
We use the same reasoning as in Equation \eqref{Equation711231} to bound
\begin{equation}
    \P\left(S_{n-1}-\floor{b_n}\ge \e_n x\right)\le C_1'' n L(\e_nx)(\e_nx)^{-\alpha}\, ,
\end{equation}
for some universal $C_1''>0$. This implies that for some universal $C_1>0$
\begin{equation}
    \mathrm{B}\le C_1 x^{-(1+\alpha)}L(x)n \left(n\frac{L(\e_n x)L(\et_n x)}{L(x)}x^{-\alpha}(\et_n)^{-\alpha-1}(\e_n)^{-\alpha}\right)\, .
\end{equation}
\textbf{3. Estimating $\mathrm{C}$}: it remains to bound the term $\mathrm{C}$, which we split further for some $c_1>0$
\begin{equation}
    \P\left(\Sh_n=x, M_n\le \et_n x\right)=\P\left(\Sh_n=x, M_n\le c_1a_n\right)+\P\left(\Sh_n=x, M_n\in (c_1a_n,\et_n x]\right)\, .
\end{equation}
The first term can be estimated using Fuk--Nagaev alone: we have that using Equation \eqref{EqFukNag} for some other $C=C(c_1)>0$
\begin{equation}
    \P\left(\Sh_n=x, M_n\le ca_n\right)=\Ocal\left(\ex^{-C\left(x/a_n\right)}\right)\, .
\end{equation}
For the last remaining term, we combine the Fuk--Nagaev inequality with the tail-estimates for the random variables themselves. Note that $ (c_1a_n,\et_n x]$ is non-empty, as $\et_nx/a_n$ diverges, see beginning of this section. Abbreviate $J^-=\log_2(1/\et_n)$ and $J^+=\log_2(c_1x/a_n)-1$. We expand
\begin{equation}
   \begin{split}
       \P&\left(\Sh_n=x, M_n\in (c_1a_n,\e_n x]\right)=\sum_{j=J^-}^{J^+}\P\left(\Sh_n=x, M_n\in (2^{-(j+1)},2^{{-j}}]x\right)\\
       &\le\sum_{j=J^-}^{J^+}\left(n\sup_{y\in (2^{-(j+1)x},2^{-j}]}\P\left(X_1=y\right)\right)\P\left(\Sh_n=x, M_n\le 2^{-j}x\right)\, .
   \end{split} 
\end{equation}
Using the tail bounds, we have that
\begin{equation}
    \sup_{y\in (2^{-(j+1)x},2^{-j}]}\P\left(X_1=y\right)=\Ocal\left(nL\left(2^{-j}x\right)\left(2^{-j}x\right)^{-(1+\alpha)}\right)=\Ocal\left(nL(x)x^{-(1+\alpha)}\left(2^{j}\right)^{2+\alpha}\right)\, ,
\end{equation}
where we used Potter's bound to see that $L\left(2^{-j}x\right)=\Ocal\left(2^jL(x)\right)$. For the event $\{\Sh_n=x, M_n\le 2^{-j}x\}$, we use Fuk--Nagaev to get that for $\alpha\in(0,2)$
\begin{equation}
    \P\left(\Sh_n=x, M_n\le 2^{-j}x\right)=\Ocal\left(2^j\right)^{-2^{j-2}}\, ,
\end{equation}
see \cite[p. 25]{berger2019notes}. For $\alpha=2$, we first note that $x\mapsto x^{-2}\sigma^2(x)$ is eventually decreasing as $\sigma^2(x)$ is slowly varying. Hence, $2^{-j}x\sigma^2\rk{2^{-j}x}\le a_n^{-2}\sigma^2(a_n)\sim n^{-1}$. Thus, we can bound
\begin{equation}
    \P\left(\Sh_n=x, M_n\le 2^{-j}x\right)=\Ocal\rk{\ex^{2^j}2^{-j2^{j+1}}}=\Ocal\left(2^j\right)^{-2^{j-2}}\, .
\end{equation}
For $\alpha>2$, we get the same bound analogously.

Combining the above bounds gives us that for some universal $C=C(c_1)>0$ (only depending on $\alpha>0$)
\begin{equation}
    \begin{split}
        \P&\left(\Sh_n=x, M_n\in (c_1a_n,\et_n x]\right)=nL(x)x^{-(1+\alpha)}\sum_{j=J^-}^{J^+}\left(2^{j}\right)^{2+\alpha}\left(c2^j\right)^{-2^{j-2}}=nL(x)x^{-(1+\alpha)}\Ocal\left(\ex^{-C\et_n^{-1}}\right)\, .
    \end{split}
\end{equation}
\textbf{4. Collection of the error bounds:} the previous calculations can be summarized as follows:
\begin{equation}
    \P(S_n=x)=n\P(X_1=x)\left(1+\mathrm{E}\right)\, ,
\end{equation}
with
\begin{equation}\label{EqLargeErrorBound}
    \mathrm{E}=\Ocal\left(\mathrm{err}[x,\e_n x]+\e_n+nL(x)x^{-(1+\alpha)}+\ex^{-C\et_n^{-1}}+n\frac{L(\e_n x)L(\et_n x)}{L(x)}x^{-\alpha}(\e_n)^{-\alpha}\et_n^{-(1+\alpha)}\right)\, .
\end{equation}
The main challenge in this case is to balance the last term in Equation \eqref{EqLargeErrorBound} with term $\e_n$. We do a case distinction, depending on the value of $\alpha$.

\textbf{The case $\mathbf{\boldsymbol{\alpha\in} (0,2)}$}: recall that by Equation \eqref{Equation725}, we have
\begin{equation}
    nx^{-\alpha}\sim \left(\frac{x}{a_n}\right)^{-\alpha}\frac{1}{L(a_n)}\, .
\end{equation}
For $\alpha\in(0,2)$, we choose $\e_n =(x/a_n)^{-\alpha_1}$, for some $\alpha_1\in (0,1)$. This gives $x=a_n\delta_n$ with $\delta_n=(x/a_n)^{1-\alpha_1}$. Note that for any $\e_1,\e_2>0$ and $C>0$ depending of $\e_1,\e_2$
\begin{equation}\label{Equation8123}
    \begin{split}
       n\frac{L(\e_n x)L(\et_n x)}{L(x)}x^{-\alpha}(\e_n)^{-\alpha}\et_n^{-(1+\alpha)}&\le C \frac{L(\e_n x)L(\et_n x)}{L(x)L(a_n)}\left(\frac{x}{a_n}\right)^{-\alpha}(\e_n)^{-\alpha}\et_n^{-(1+\alpha)}\\
        &\le C \frac{L(\e_n x)L(\et_n x)}{L(x)L(a_n)}(\delta_n)^{-\alpha}\et_n^{-(1+\alpha)}\\
        &\le C(\delta_n)^{-\alpha+\e_1}\et_n^{-(1+\alpha)+\e_2}\, .
    \end{split}
\end{equation}
where we used the Potter bounds in Equation \eqref{eqrefpot} twice, once with $\delta=\e_1$ and once with $\delta=\e_2$:
\begin{equation}
    \frac{L(\e_n x)}{L(a_n)}\le \delta_n^{\e_1}\quad\text{and}\quad \frac{L(\et_n x)}{L(x)}\le \et_n^{-\e_2}\, .
\end{equation}
Recall that $\et_n=\log(x/a_n)^{-2}$.  We obtain (for some $\e_3,\e_3'$ which can be made arbitrarily small, as $\e_1,\e_2$ becomes small)
\begin{equation}
      n\frac{L(\e_n x)L(\et_n x)}{L(x)}x^{-\alpha}(\e_n)^{-\alpha}\et_n^{-(1+\alpha)}\le C (\delta_n)^{-\alpha+\e_3}\le C\left(\frac{x}{a_n}\right)^{-(1-\alpha_1)\alpha-\e_3'}\, .
\end{equation}
For $\alpha_1=\alpha/(\alpha+1)$, both terms are approximately equal and we hence obtain
\begin{equation}
    \e_n+\left(\frac{x}{a_n}\right)^{-(1-\zeta)\alpha-\e_3'}=\left(\frac{x}{a_n}\right)^{-\zeta}+\left(\frac{x}{a_n}\right)^{-(1-\zeta)\alpha-\e_3'}\le 2\left(\frac{x}{a_n}\right)^{-\frac{\alpha}{\alpha+1}+\e''}\, ,
\end{equation}
for some $\e''=o(1)$ as $\e_1,\e_2\downarrow 0$.

The previous equation reduces the error in Equation \eqref{eqrefpot} (as the other terms are negligible) to 
\begin{equation}
    \mathrm{E}=\Ocal\left(\mathrm{err}[x,a_n\delta_n]+2\left(\frac{a_n}{x}\right)^{\alpha_1-\e''}\right)\, ,
\end{equation}
with $\e''>0$ as small as we want. This concludes the proof of Theorem \ref{Thm1} for the case $\alpha\in(0,2)$.

\textbf{The case $\mathbf{\boldsymbol{\alpha\in} (2,\infty)}$}: recall $\alpha_1=\alpha/(1+\alpha)$. Choose the largest possible $\beta\ge 0$ such that
\begin{equation}
    n^{-\beta}\rk{\frac{x}{\sqrt{n\log(n)}}}^{1-\alpha_1}\to\infty\qquad\text{and}\qquad \beta\le \frac{(\alpha-2)(\alpha+1)}{2(2\alpha+1)}\, .
\end{equation}
Choose $\e_n=n^{-\beta}\rk{\frac{x}{\sqrt{n\log(n)}}}^{-\alpha_1}$. Note that this allows us to rewrite
\begin{equation}
    nx^{-\alpha}\e_n^{-\alpha}=n^{1-\alpha/2}\rk{\frac{x}{\sqrt{n\log(n)}}}^{-\alpha}\e_n^{-\alpha}\log^{-\alpha/2}(n)=n^{1-\alpha/2+\beta\alpha_1}\rk{\frac{x}{\sqrt{n\log(n)}}}^{-\alpha(1-\alpha_1)}\!\!\!\!\!\!\log^{-\alpha/2}(n)\, .
\end{equation}
Note that for the choices of $\alpha_1,\beta$, we have that $-\alpha_1=-\alpha(1-\alpha_1)$ and $1-\alpha/2+\beta\alpha_1\le -\beta$. Hence, we get that
\begin{equation}
\begin{split}
    \Ocal\rk{nx^{-\alpha}\e_n^{-\alpha}+\e_n}&=\Ocal\rk{n^{1-\alpha/2+\beta\alpha_1}\rk{\frac{x}{\sqrt{n\log(n)}}}^{-\alpha(1-\alpha_1)}\!\!\!\!\!\!\log^{-\alpha/2}(n)+n^{-\beta}\rk{\frac{x}{\sqrt{n\log(n)}}}^{-\alpha_1}}\\
    &=\Ocal\rk{n^{1-\alpha/2+\beta\alpha_1}\rk{\frac{x}{\sqrt{n\log(n)}}}^{-\alpha_1}}\, .
\end{split}
\end{equation}
As in the case $\alpha\in(0,2)$, the slowly varying functions add at most a power of $\e''>0$, where we can choose $\e''>0$ as small as we desire. Furthermore, $nx^{-(1+\alpha)}L(x)=o\rk{nx^{-\alpha}\e_n^{-\alpha}}$. This gives
\begin{equation}
    \mathrm{E}=\Ocal\rk{n^{1-\alpha/2+\beta\alpha_1}\rk{\frac{x}{\sqrt{n\log(n)}}}^{-\alpha_1+\e''}}
\end{equation}
The positive $\e''$ easily absorbs the decay of order $\log(n)$. Hence, we can conclude the proof as we did in the case $\alpha\in (0,2)$.

\textbf{The case $\mathbf{\boldsymbol{\alpha=} 2}$}: 
% choose $n\mapsto M(n)$ slowly varying such $a_n=\sqrt{n}M(n)$. Note that $M(n)\to\infty$, as $n\to\infty$. 
Choose $\e_n=\rk{\frac{x}{a_n\sqrt{\log(n)}}}^{-\alpha_1}$, for $\alpha_1=2/3$. We then have that
\begin{equation}
     nx^{-2}\e_n^{-2}=\rk{\frac{x}{a_n\sqrt{\log(n)}}}^{-\alpha_1}\rk{\frac{\sqrt{n}}{a_n\sqrt{\log(n)}}}^2\, .
\end{equation}
This gives that
\begin{equation}
    \Ocal\rk{nx^{-2}\e_n^{-2}+\e_n}=\Ocal\rk{\rk{\frac{x}{a_n\sqrt{\log(n)}}}^{-\alpha_1}}\, .
\end{equation}
From there on, we proceed as in the case $\alpha>2$, noting that $\beta=0$.
\qed

\begin{remark}
    We expect that error calculated above to be essentially optimal (up to the $\e>0$ which can be chosen as small as we want. Indeed, probabilistically, there are two sources of errors: the maximum can deviate from $x$ by $\e_nx$. This gives an error of $\Ocal\rk{\e_n}$. This error shrinks as we make $\e_n$ small. However, the remaining sum compensating by being larger/smaller than their CLT scale gives an error of $\Ocal\rk{(x/a_n)^{-\alpha}\e_n^{-\alpha}}$, which shrinks as we increase $\e_n$, see Equation \eqref{Equation8123}. Both error terms are optimal in the sense that we cannot replace $\Ocal(\ldots)$ by $o(\ldots)$, see \cite[Theorem 2.1]{berger2019notes}. Our choice of $\e_n$ makes the two errors asymptotically equal, selecting the minimal possible error. 
\end{remark}
\subsection{Proof of Theorem \ref{ThmNonLocal}}
The proof of Theorem \ref{ThmNonLocal} follows the same steps as the one of Theorem \ref{Thm1}: we first split the probability
\begin{equation}
      \P(\Sh_n\ge x)=\P\left(\Sh_n\ge x, M_n\ge (1-\e_n)x\right)+\P\left(\Sh_n\ge x, M_n< (1-\e_n)x\right)\, .
\end{equation}
The second term is negligible and produces the same errors as the terms $\mathrm{B}$ and $\mathrm{C}$ in the proof of Theorem \ref{Thm1}.

We upper bound the first term
\begin{equation}
    \P\left(\Sh_n\ge x, M_n\ge (1-\e_n)x\right)\le n\P\rk{X_1>(1-\e_n)x}=n\P\rk{X_1>x}\rk{1+\mathrm{err}[\e_n ,x]+\Ocal(\e_n)}\, .
\end{equation}
The lower bound is analogous to Equation \eqref{EquationOmitted} and is hence omitted. This concludes the proof of Theorem \ref{ThmNonLocal}.\qed

\subsection{Proof of Theorem \ref{Thm2} and Theorem \ref{Thm3}}
In this section we prove Theorem \ref{Thm2} and Theorem \ref{Thm3}. Theorem \ref{Thm2} will be proved in full detail while for Theorem \ref{Thm3} we just highlight the differences with Theorem \ref{Thm2}.

Set $x^-=x-\omega(a_n)$ for $\omega(a_n)>0$ fixed  . Recall that
\begin{equation}
    \nu_x=\P(X_1\in \,\cdot\,|X_1>x-\omega(a_n))\qquad\text{and}\qquad \nu_{x,n}=\P\left(\{X_i\}_{i=1,\ldots,n}\in\,\cdot\,|S_n>x\right)\, .
\end{equation}
Let
\begin{equation}
    \nu_{x,n}^*=\frac{1}{n}\sum_{j=1}^n\sigma^j\left(\mu^{\otimes (n-1)}\otimes \nu_x\right)\, ,
\end{equation}
where $\sigma^j$ switches the last coordinate with the $j$-th coordinate. We then have that using Pinsker’s inequality and Csisz\'{a}r’s parallelogram identity (see \cite{armendariz2011conditional})
\begin{equation}
    \norm{\nu_{x,n}-\nu_{x,n}^*}^2\le H(\nu_{x,n}|\mu^{\otimes n})+H(\nu_{x,n}^*|\mu^{\otimes n})-2H\left(\frac{\nu_{x,n}+\nu_{x,n}^*}{2}\Big|\mu^{\otimes n}\right)=\mathrm{A}+\mathrm{B}-\mathrm{C}\, ,
\end{equation}
where
\begin{equation}
    H(\mu|\nu)=\begin{cases}
    \int \frac{\d\mu}{\d\nu}\log  \left[\frac{\d\mu}{\d\nu}\right]\d \nu&\text{ if }\mu\ll\nu\, ,\\
    +\infty&\text{ otherwise.}
    \end{cases}
\end{equation}
Note that for $y\in \R^n$
\begin{equation}
    \frac{\d\nu_{x,n}}{\d\mu^{\otimes n}}(y)=\frac{\1\{S_n(y)>x\}}{G_n(x)}\, ,
\end{equation}
where $G_n(x)=\P(S_n>x)$ and $S_n(y)=\sum_{i=1}^ny_i$. Note that
\begin{equation}
    \frac{\d\nu_{x,n}^*}{\d\mu^{\otimes n}}(y)=\frac{1}{nG(x^-)}\sum_{i=1}^n\1\{y_i>x^-\}\, 
\end{equation}
where we recall $G(t)=\P(X_1>t)$.\\
We have that
\begin{equation}
    \mathrm{A}+\mathrm{B}=H(\nu_{x,n}|\mu^{\otimes n})+H(\nu_{x,n}^*|\mu^{\otimes n})=\int\log N_{x,n}\d\nu_{x,n}-\log\left({G}_n(x)n{G}(x^-)\right) \, ,
\end{equation}
where $N_{x,n}$ counts the number of coordinates larger than $x^-$. Note that
\begin{equation}
   \int\log N_{x,n}\d\nu_{x,n}=\sum_{k\ge 2}\log(k) \nu_{x,n}(N_{x,n}=k)=\sum_{k\ge 2}\log(k) \binom{n-1}{k-1}G(x^-)^{k-1}=\Ocal\left(nG(x)\right)\, .
\end{equation}
Thus
\begin{equation}\label{EquationAplusB}
    \mathrm{A}+\mathrm{B}=-\log\left({G}_n(x)n{G}(x^-)\right)+\Ocal\left(nG(x)\right)\, .
\end{equation}
On the other hand,
\begin{equation}
    \mathrm{C}=2H\left(\frac{\nu_{x,n}+\nu_{x,n}^*}{2}\Big|\mu^{\otimes n}\right)=\int \left[\frac{\d\nu_{x,n}}{\d\mu^{\otimes n}}+\frac{\d\nu_{x,n}^*}{\d\mu^{\otimes n}}\right]\log\left[\frac{\d\nu_{x,n}}{2\d\mu^{\otimes n}}+\frac{\d\nu_{x,n}^*}{2\d\mu^{\otimes n}}\right]\d\mu^{\otimes n}\, .
\end{equation}
We split the integrand into two: for the first part, we estimate
\begin{equation}
\begin{split}
    \int&\frac{\nu_{x,n}^*}{\d\mu^{\otimes n}} \log\left[\frac{\d\nu_{x,n}}{2\d\mu^{\otimes n}}+\frac{\d\nu_{x,n}^*}{2\d\mu^{\otimes n}}\right]\d\mu^{\otimes n}=\int \frac{N_{x,n}}{n{G}(x^-)}\log\left[\frac{\d\nu_{x,n}^*}{2\d\mu^{\otimes n}}+\frac{1}{2{G}_n(x)}\right]\d\mu^{\otimes n}\\
    =&\int \frac{\1_{N_{x,n}=1}}{n{G}(x^-)}\log\left[ \frac{1}{2 n{G}(x)}+\frac{1}{2{G}_n(x)}\right]\d\mu^{\otimes n}+\int \frac{N_{x,n}\1_{N_{x,n}>1}}{n{G}(x^-)}\log\left[\frac{\d\nu_{x,n}^*}{2\d\mu^{\otimes n}}+\frac{1}{2{G}_n(x)}\right]\d\mu^{\otimes n}\, .
\end{split}
\end{equation}

Note that by the inclusion-exclusion principle
\begin{equation}
    \mu^{\otimes n}\rk{N_{x,n}=1}=nG(x^-)+\sum_{k=2}^n(-1)^{k-1}\binom{n}{k}G(x^-)^k=nG(x^-)\rk{1+\Ocal\rk{nG(x^-)}}\, .
\end{equation}
Hence,
\begin{multline}\label{Equationalimentary}
\int \frac{\1_{N_{x,n}=1}}{n{G}(x^-)}\log\left[ \frac{1}{2 n{G}(x^-)}+\frac{1}{2{G}_n(x)}\right]\d\mu^{\otimes n}=\left(1+\Ocal\rk{G(x^-)n}\right)\log\left[ \frac{1}{2 n{G}(x^-)}+\frac{1}{2{G}_n(x)}\right]\\
    =-\log\left[nG(x)\right]\left(1+\Ocal\left(\mathrm{err}\hk{1}[x,\omega(a_n)]+c_{x,n}+G(x^-)n\right)\right)\, .
\end{multline}
Indeed,
\begin{equation}
    {G}(x^-)=G(x)\left(1+\mathrm{err}\hk{1}[x,\omega(a_n)]\right)\, ,
\end{equation}
and
\begin{equation}
    {G}_n(x)=nG(x)\left(1+\Ocal(c_{x,n})\right)\, .
\end{equation}
The error term is given by
\begin{equation}
    \begin{split}
       \int \frac{N_{x,n}\1_{N_{x,n}>1}}{n{G}(x)}\log\left[\frac{\d\nu_{x,n}^*}{2\d\mu^{\otimes n}}+\frac{1}{2{G}_n(x)}\right]\d\mu^{\otimes n}\le C\sum_{k=2}^n \left[n{G}(x)\right]^{k-1}k\binom{n}{k}\log[n{G}(x)]\, ,
    \end{split}
\end{equation}
and thus (noting that the term $k=2$ dominates)
\begin{equation}
     \int\frac{\d\nu_{x,n}}{\d\mu^{\otimes n}} \log\left[\frac{\d\nu_{x,n}}{2\d\mu^{\otimes n}}+\frac{\d\nu_{x,n}^*}{2\d\mu^{\otimes n}}\right]\d\mu^{\otimes n}=-\log\left[nG_n(x)\right]\left(1+\Ocal\left(\mathrm{err}\hk{1}[x,\omega(a_n)]+c_{x,n}+n{G}(x)\right)\right)\, .
\end{equation}
For the second term, note that
\begin{equation}
    \int \frac{\d\nu_{x,n}^*}{\d\mu^{\otimes n}}\log\left[\frac{\d\nu_{x,n}}{2\d\mu^{\otimes n}}+\frac{\d\nu_{x,n}^*}{2\d\mu^{\otimes n}}\right]\d\mu^{\otimes n}= \int \frac{\1\{S_n>x\}}{{G}_n(x)}\log\left[\frac{N_{x,n}}{2n{G}(x)}+\frac{1}{2{G}_n(x)}\right]\d\mu^{\otimes n}\, .
\end{equation}
Note that
\begin{equation}
    \int \frac{\1\{S_n>x, \, N_{x,n}=0\}}{{G}_n(x)}\log\left[\frac{N_{x,n}}{2n{G}(x)}+\frac{1}{2{G}_n(x)}\right]\d\mu^{\otimes n}=\Ocal\left({c_{n,x}\log\left(nG(x)\right)}\right)\, .
\end{equation}
On the other hand,
\begin{equation}
    \int \frac{\1\{S_n>x, \, N_{x,n}\ge 2\}}{{G}_n(x)}\log\left[\frac{N_{x,n}}{2n{G}(x)}+\frac{1}{2{G}_n(x)}\right]\d\mu^{\otimes n}=\Ocal\left(\log\left(nG(x)\right)nG(x)\right)\, ,
\end{equation}
similar to before. We estimate the final contribution
\begin{equation}
   \int \frac{\1\{S_n>x, \, N_{x,n}=1\}}{{G}_n(x)}\log\left[\frac{1}{2n{G}(x)}+\frac{1}{2{G}_n(x)}\right]\!\d\mu^{\otimes n}\!=-\log\left[G_n(x)\right]\left(1+\Ocal\left(\mathrm{err}\hk{1}[x,\omega(a_n)]+c_{x,n}\right)\right) .
\end{equation}
This is done analogously to Equation \eqref{Equationalimentary}. Combining the above bounds yields that
\begin{equation}
    \mathrm{C}=-2\log\left[G_n(x)\right]\left(1+\Ocal\left(\mathrm{err}\hk{1}[x,\omega(a_n)]+c_{x,n}+n{G}(x)\right)\right) \, .
\end{equation}
As $\norm{\nu_{x,n}-\nu_{x,n}^*}^2\le \mathrm{A}+\mathrm{B}-\mathrm{C}$, we have using Equation \eqref{EquationAplusB}
\begin{equation}
     \norm{\nu_{x,n}-\nu_{x,n}^*}^2\le\Ocal\left(\mathrm{err}\hk{1}[x,\omega(a_n)]+c_{x,n}+n{G}(x)\right)\, .
\end{equation}
This concludes the proof of Theorem \ref{Thm2}.\qed

The proof of Theorem \ref{Thm3} can now be carried out in exactly the same manner: note that for $y\in\R^n$
\begin{equation}
    \frac{\d\xi_{x,n}^*}{\d\mu^{\otimes n}}(y)=\frac{\1\{S_{n-1}(y)=x-y_n\}}{\P(X_n=S_{n-1}-x)}\, .
\end{equation}
By conditioning on $X_n$, we can establish
\begin{equation}
  \P(X_n=S_{n-1}-x)=nG(x)\left(1+\mathrm{err}\hk{2}[\omega(x,a_n)]+c_{n,x}\right)\, .
\end{equation}
On the other hand,
\begin{equation}
     \frac{\d\xi_{x,n}}{\d\mu^{\otimes n}}(y)=\frac{\1\{S_{n}(y)=x\}}{\P(S_n=x)}\, .
\end{equation}
We can now apply the error bounds from Theorem \ref{Thm1} together with the method from the proof of Theorem \ref{Thm2} to conclude the proof.
\section*{Acknowledgements}
The author would like to express his gratitude for the anonymous referee for suggesting several improvements for this paper, most notably the extension from $\alpha\in(0,2)$ to $\alpha>0$. The author would also like the thank the referee for pointing out a calculation mistake in an earlier version of the paper. The author would like to thank Quentin Berger for his help, answering my questions both quickly and patiently. The author would also like to thank Silke Rolles and Julius Damarackas for their help regarding typos and presentation.

\end{document}